\documentclass{amsart}  

\usepackage{amsmath, latexsym, amssymb, stmaryrd}
\usepackage{amscd}
\usepackage{color}
\usepackage{amsmath, amsfonts, amssymb, hyperref}
\usepackage{verbatim}
\usepackage{enumerate}
\usepackage[all]{xy}

\definecolor{gr75}{gray}{0.75}

\theoremstyle{plain}
\newtheorem{theorem}{Theorem}[section]
\newtheorem*{theorem*}{Theorem}

\newtheorem{definition}[theorem]{Definition}
\newtheorem{example}[theorem]{Example}

\theoremstyle{remark}
\newtheorem{remark}[theorem]{Remark}

\numberwithin{equation}{section}
\newcommand{\irr}{\text{Irr}}
\newcommand{\ind}{\text{Ind}\!\uparrow}

\newcommand{\tr}{\text{tr}}
\newcommand{\ch}{\text{ch}}

\newcommand{\bC}{\mathbb{C}}

\newcommand{\bK}{\mathbb{K}}

\newcommand{\bZ}{\mathbb{Z}}

\newcommand{\Sym}{\textsl{Sym}}

\newcommand{\Qsym}{\textsl{QSym}}
\newcommand{\qs}{\check{\mathcal{S}}}	
\newcommand{\qsy}{\hat{\mathcal{S}}}	

\newcommand{\Nsym}{\textsl{NSym}}
\newcommand{\nch}{\mathbf{h}}         
\newcommand{\nce}{\mathbf{e}}         
\newcommand{\ncb}{\mathbf{b}}  
\newcommand{\ncr}{\mathbf{r}}         
\newcommand{\ncs}{\check{\mathbf{s}}} 
\newcommand{\ncsy}{\hat{\mathbf{s}}}            

\newcommand{\dalg}{\mathcal{D}}   
\newcommand{\dah}{B}                       
\newcommand{\dar}{D}                         

\newcommand{\palg}{\mathcal{P}}   

\newcommand{\ralg}{\mathcal{R}}

\newcommand{\calg}{\mathcal{C}\ell}
\newcommand{\ych}[1]{\xi^{#1}}         

\newcommand{\sgrp}{\mathfrak{S}}

\newcommand\partitionof[1]{\widetilde{#1}}
\newcommand\reverse[1]{{#1}^*}

\newcommand{\refines}{\preccurlyeq}   

\newcommand{\suchthat}{\;|\;}

\newcommand{\set}{\mathrm{Set}} 
\newcommand{\des}{\mathrm{Des}} 
\newcommand{\comp}{\mathrm{comp}} 
 
\newcommand{\rhocy}{{\hat{\rho}}} 

\newlength\cellsize 
\savebox2{%
\begin{picture}(15,15)
\put(0,0){\line(1,0){15}}
\put(0,0){\line(0,1){15}}
\put(15,0){\line(0,1){15}}
\put(0,15){\line(1,0){15}}
\end{picture}}
\newcommand\cellify[1]{\def\thearg{#1}\def\nothing{}%
\ifx\thearg\nothing
\vrule width0pt height\cellsize depth0pt\else
\hbox to 0pt{\usebox2\hss}\fi%
\vbox to 15\unitlength{
\vss
\hbox to 15\unitlength{\hss$#1$\hss}
\vss}}
\newcommand\tableau[1]{\vtop{\let\\=\cr
\setlength\baselineskip{-16000pt}
\setlength\lineskiplimit{16000pt}
\setlength\lineskip{0pt}
\halign{&\cellify{##}\cr#1\crcr}}}
\savebox3{%
\begin{picture}(15,15)
\put(0,0){\line(1,0){15}}
\put(0,0){\line(0,1){15}}
\put(15,0){\line(0,1){15}}
\put(0,15){\line(1,0){15}}
\end{picture}}
\newcommand\expath[1]{%
\hbox to 0pt{\usebox3\hss}%
\vbox to 15\unitlength{
\vss
\hbox to 15\unitlength{\hss$#1$\hss}
\vss}}
\newcommand\bas[1]{\omit \vbox to \cellsize{ \vss \hbox to \cellsize{\hss$#1$\hss} \vss}}

\begin{document}

\bibliographystyle{plain}

\title[noncommutative irreducible characters of the symmetric group]{Noncommutative irreducible characters of the symmetric group  and noncommutative Schur functions}

\author{S. van Willigenburg}
\address{Department of Mathematics, University of British Columbia, Vancouver, BC V6T 1Z2, Canada}
\email{\href{mailto:steph@math.ubc.ca}{steph@math.ubc.ca}}
\thanks{
The author was supported in part by the National Sciences and Engineering Research Council of Canada.}

\begin{abstract}
In the Hopf algebra of symmetric functions, $\Sym$, the basis of Schur functions is distinguished since every Schur function is isomorphic to an irreducible character of a symmetric group under the Frobenius characteristic map. In this note we show that in the Hopf algebra of noncommutative symmetric functions, $\Nsym$, of which $\Sym$ is a quotient, the recently discovered basis of noncommutative Schur functions exhibits that every \emph{noncommutative} Schur function is isomorphic to a \emph{noncommutative} irreducible character of a symmetric group when working in noncommutative character theory. We simultaneously show that a second basis of $\Nsym$ consisting of Young noncommutative Schur functions also satisfies that every element is isomorphic to a {noncommutative} irreducible character of a symmetric group.
\end{abstract}

\maketitle

\section{Introduction}\label{sec:intro} In the representation theory of the symmetric group it is long-established \cite{Schur} that every Schur function is isomorphic to an irreducible character of a symmetric group under the Frobenius characteristic map. Since this was established, Schur functions have continued to be a vibrant area of research due to their appearance in areas such enumerative combinatorics through Young tableaux, plus areas such as algebraic geometry and diagonal harmonics. Throughout, Schur functions have exhibited beautiful combinatorial properties such as Pieri rules, the Littlewood-Richardson rule, and Kostka numbers.

Schur functions also form a basis for the Hopf algebra of symmetric functions, $\Sym$, which is contained in the Hopf algebra of quasisymmetric functions, $\Qsym$. Recently  a new basis for $\Qsym$ has been discovered \cite{HLMvW-1}, using the combinatorics of Macdonald polynomials \cite{HHL-1}. This new basis was termed the basis of quasisymmetric Schur functions as quasisymmetric Schur functions refine Schur functions in a natural way. Furthermore, both the basis of quasisymmetric Schur functions and their dual basis, the basis of noncommutative Schur functions in the Hopf algebra of noncommutative symmetric functions $\Nsym$, exhibit many natural generalisations of combinatorial properties of Schur functions such as Pieri rules \cite[Theorem 6.3]{HLMvW-1}, the Littlewood-Richardson rule \cite{BLvW, HLMvW-2} and Kostka numbers \cite[Theorem 6.1]{HLMvW-1}.

Therefore a pertinent question to ask is can every noncommutative Schur function be realized as being isomorphic to an irreducible character of a symmetric group? In this note we answer the question in the affirmative in Theorem~\ref{the:noncommirrchars} using \emph{noncommutative} character theory.

The theory of noncommutative character theory for the symmetric groups was initiated in \cite{solomon}. It was then developed independently by Malvenuto and Reutenauer \cite{malv-reut-1} and Gelfand et al. \cite{GKLLRT}, and brought together by J\"ollenbeck \cite{jollenbeck-1}. Other papers on the subject have also been written including \cite{baumann-hohlweg,bless-hohl-schock,bonnafe-hohlweg, NCSF6, garsia-reut, Leclerc94noncommutativecyclic, mantaci-reut,poirier-reut}.
This theory is more amenable to combinatorial manipulation involving permutations, and many classical results such as the Murnaghan-Nakayama rule, the Littlewood-Richardson rule and results on Foulkes characters are more straightforward to obtain using this theory.
An exposition can be found in the monograph \cite{bless-schock} by Blessenohl and Schocker, which will serve as our primary reference for this theory later.

With this in mind, we will spend the remainder of this note working towards establishing Theorem~\ref{the:noncommirrchars}, which in essence states the following.

\begin{theorem*} Every noncommutative Schur function is isomorphic to a noncommutative irreducible character of a symmetric group.
\end{theorem*}

At the same time, we will establish the following for another basis of $\Nsym$, called the basis of Young noncommutative Schur functions, whose combinatorics is related to that of Young tableaux.

\begin{theorem*} Every Young noncommutative Schur function is isomorphic to a noncommutative irreducible character of a symmetric group.
\end{theorem*}

\subsection*{Acknowledgements} The author would like to thank Kurt Luoto for helpful conversations.

\section{Combinatorial tools}\label{sec:combstuff}

We begin by recalling the combinatorial concepts of compositions and partitions, which will be used to index the variety of functions we will subsequently encounter. A \emph{composition} is a finite list of positive integers, while a \emph{partition} is a finite unordered list of positive integers, which we list in weakly decreasing order for convenience. In both cases the integers are called \emph{parts}. We say a composition (respectively, partition) $\alpha$ is a composition (respectively, partition) of $n$, denoted by $\alpha \vDash n$ (respectively, $\alpha \vdash n$), if the parts of $\alpha$ sum to $n$. We denote by $0$  both the empty composition and empty partition of $0$. The \emph{underlying partition} of a composition $\alpha$, denoted by $\partitionof{\alpha}$, is the partition obtained from $\alpha$ by reordering the parts of $\alpha$ into weakly decreasing order. 

One composition closely related to a composition $\alpha = (\alpha _1, \ldots , \alpha_\ell)$ is the \emph{reversal} of $\alpha$, denoted by $\reverse{\alpha}$, given by $\reverse{\alpha}=(\alpha _\ell, \ldots , \alpha _1)$. We say a composition $\beta$ \emph{refines} a composition $\alpha$ if we can obtain the parts of $\alpha$ in order by adding together adjacent parts of $\beta$ in order, denoted by $\beta \refines \alpha$. For any composition $\alpha \vDash n$, say $\alpha =(\alpha_1,\ldots,\alpha_\ell)$, we associate the unique subset $\set(\alpha)\subseteq[n-1] = \{1, 2, \ldots ,n-1\}$ given by the partial sums of $\alpha$, that is, 
\[ \set(\alpha) = \left\{ \sum_{i=1}^k \alpha_i \suchthat 1\leqslant k < \ell \right\}. \]
Conversely, for $J\subseteq [n-1]$ we write $\comp(J)=\alpha$ if $\set(\alpha)=J$.

\begin{example}\label{ex:comp}
If  $\alpha=(2,4,1,2)$, then $\partitionof{\alpha}=(4,2,2,1)$,  $\reverse{\alpha}=(2,1,4,2)$, $(2,1,1,3,1,1)\refines (2,1,4,2)$ and  $\set(\alpha)=\{2,6,7\}\subset[8]$. 
\end{example}


Now we use compositions to create diagrams.

\begin{definition}\label{def:compdiags}Given a composition $\alpha=(\alpha _1, \ldots , \alpha _\ell)\vDash n$, we say the \emph{composition diagram} of $\alpha$, also denoted by $\alpha$, is the left-justified array of $n$ cells with $\alpha_i$ cells in the $i$-th row from the bottom. The cell in the $i$-th row from the bottom and $j$-th cell from the left is denoted by the pair $(i,j)$. 
\end{definition}

Our next task is to fill our composition diagrams with integers.

\begin{definition}\label{def:SYCT} Given a composition $\alpha$, a \emph{standard Young composition tableau} (abbreviated to SYCT) $\tau$ of \emph{shape} $\alpha$, denoted by $sh(\tau)=\alpha$, is a bijective filling of the cells of $\alpha$
$$\tau: \alpha \to [n]$$such that
\begin{enumerate}
\item the entries in each row are increasing when read from left to right,
\item the entries in the first column are increasing when read from the row with the smallest index to the row with the largest index,
\item if $i<j$ and $\tau(j,k)<\tau(i, k+1)$ then $\tau(j, k+1)$ exists and $\tau(j,k+1)<\tau(i,k+1)$.
\end{enumerate}\end{definition}
We use SYCTs to denote the set of all SYCTs where the context is clear.
The \emph{descent set} of an SYCT $\tau$ is defined to be
$$\des(\tau) = \{i \suchthat i+1 \mbox{ appears in the same column or a column to the left of } i\}.$$Given an SYCT $\tau$ with $n$ cells, we let $\comp(\tau)=\comp(\des(\tau))\vDash n$ and call it the \emph{descent composition} of $\tau$.

\begin{example}\label{ex:SYCT} An SYCT $\tau$ with $sh(\tau)= (2,4,1,2)$, $\des(\tau)=\{1,3,4,6\}$ and $\comp(\tau)=(1,2,1,2,3)$.
$$\tableau{7&8\\5\\2&3&6&9\\1&4}$$
\end{example}

\begin{remark} There is a natural bijection $\rhocy$ from the set of all SYCTs to standard Young tableaux that arise in the classical representation theory of the symmetric group: Given an SYCT $\tau$ we obtain $\rhocy(\tau)$ by sorting the entries in each column from bottom to top in increasing order and then bottom justifying the columns. For further details, see \cite{LMvW-book}. This bijection is analogous to the bijection due to Mason \cite{mason-1} for standard reverse tableaux.
\end{remark}


It is well-known that the symmetric group $\sgrp _n$ is the group of bijections $\sigma: [n] \to [n]$, called permutations, which we will denote in one-line notation $\sigma(1)\cdots\sigma(n)$. The symmetric group $\sgrp_n$ is also known as the Coxeter group of type $A$, $A_{n-1}$, whose standard generators are 
$s_i=1\cdots (i-1)(i+1)i(i+2)\cdots n$ for $1\leqslant i \leqslant n-1$.

The \emph{descent set} of a permutation $\sigma \in \sgrp _n$ is defined to be
$$\des(\sigma)=\{i\in [n-1] \suchthat \sigma(i)>\sigma(i+1)\}.$$In terms of Coxeter groups $i\in \des(\sigma)$ if and only if there is some minimal length word in terms of standard generators for $\sigma$ that ends in $s_i$. Given a permutation $\sigma \in \sgrp _n$ we let $\comp(\sigma)=\comp(\des(\sigma))\vDash n$ and call it the \emph{descent composition} of $\sigma$.

We now review a number of algebraic concepts. Throughout we let $\bK$ denote our coefficient ring, 
which under minimal conditions we can take to be $\bZ$ 
but which we ordinarily take to be $\bC$.

\section{The Hopf algebras of symmetric and quasisymmetric functions}\label{sec:SymandQsym}

We now recap some pertinent facts about the Hopf algebra of symmetric functions. Further details can be found in excellent texts such as \cite{fulton-1, macdonald-1, sagan, stanley-ec2}. The Hopf algebra of symmetric functions is a graded Hopf algbera
$$\Sym = \bigoplus _{n\geqslant 0} \Sym ^n \subseteq \bK[[x_1, x_2, \ldots ]]$$and the fundamental theorem of symmetric functions states that $\Sym$ is a polynomial algebra in the elementary symmetric functions:
$$\Sym \cong \bK[e_1, e_2, \ldots ]$$where the \emph{$r$-th elementary symmetric function} for $r>0$ is
$$e_r=\sum _{1\leqslant i_1 <i_2 <\cdots < i_r} x_{i_1}x_{i_2}\cdots x_{i_r}.$$It is also a polynomial algebra in the complete homogeneous symmetric functions:
$$\Sym \cong \bK[h_1, h_2, \ldots ]$$where the \emph{$r$-th complete homogeneous symmetric function} for $r>0$ is
$$h_r=\sum _{1\leqslant i_1 \leqslant i_2 \leqslant \cdots \leqslant i_r} x_{i_1}x_{i_2}\cdots x_{i_r} = \sum _{  (\alpha _1, \ldots, \alpha _\ell) \vDash r} (-1)^{\ell-r} e _{\alpha _ 1}\ldots e_{\alpha _\ell}.$$To obtain $\bK$-bases for $\Sym$ define for a partition $\lambda =(\lambda _1, \ldots , \lambda _\ell)$ the \emph{elementary} and \emph{complete homogeneous symmetric functions} to be, respectively,
$$e_\lambda = e_{\lambda _1}\cdots e_{\lambda _\ell}$$
$$h_\lambda = h_{\lambda _1}\cdots h_{\lambda _\ell}.$$From here the fundamental theorem of symmetric functions implies that $\Sym ^n$ for $n>0$ has either $\{e_\lambda\} _{\lambda \vdash n}$ or $\{h_\lambda\} _{\lambda \vdash n}$ as a $\bK$-basis. We define $e_0=h_0=1$ and $e_r=h_r=0$ for $r<0$.

However, the most renowned $\bK$-basis for $\Sym$ is that consisting of Schur functions, which can be defined in a myriad of ways including as quotients, as irreducible representations of a symmetric group, or as generating functions for tableaux, but for our purposes we will define them using the following
\emph{Jacobi-Trudi determinant}. 

\begin{definition}\label{def:Schur}Given a partition $\lambda = (\lambda _1, \ldots , \lambda _\ell)$ the \emph{Schur function} $s_\lambda$ is defined to be   $s_0=1$ and otherwise
$$s_\lambda = \det (h_{\lambda _i -i+j})_{i,j=1}^\ell.$$\end{definition}
Again we have $\{s_\lambda\} _{\lambda \vdash n}$ is a $\bK$-basis for $\Sym ^n$.

The Hopf algebra of symmetric functions $\Sym$ is contained in the graded Hopf algebra of quasisymmetric  functions $\Qsym$
$$\Qsym = \bigoplus _{n\geqslant 0} \Qsym ^n \subseteq \bK[[x_1, x_2, \ldots ]]$$where $\Qsym ^n$ for $n>0$ is spanned by the $\bK$-basis of \emph{monomial quasisymmetric functions} $\{ M_\alpha \} _{\alpha \vDash n}$ given by
$$M_{\alpha}=\sum _{i_1<\cdots < i_\ell}x_{i_1}^{\alpha_1}\cdots x_{i_\ell}^{\alpha_\ell}$$and is also spanned by the $\bK$-basis of \emph{fundamental quasisymmetric functions} $\{ F_\alpha \} _{\alpha \vDash n}$ given by
$$F_{\alpha}=\sum_{\beta\preccurlyeq\alpha}M_{\beta}$$and $M_0=F_0=1$.

In \cite{LMvW-book} another $\bK$-basis for $\Qsym ^n$ was introduced, that of {Young quasisymmetric Schur functions} $\{ \qsy _\alpha \} _{\alpha \vDash n}$ defined combinatorially as follows.

\begin{definition}\label{def:qsy}\cite[Proposition 5.2.2]{LMvW-book} Let $\alpha \vDash n$. Then the \emph{Young quasisymmetric Schur function} $\qsy _\alpha$ is given by $\qsy _0=1$ and otherwise
\begin{equation}\label{eq:qsy}\qsy _\alpha = \sum _\beta \hat{d}_{\alpha \beta} F_\beta\end{equation}where the sum is over all compositions $\beta \vDash  n$ and $\hat{d}_{\alpha\beta}=$ the number of $SYCTs$ $\tau$ of shape $\alpha$ such that $\des (\tau)=\set (\beta)$. 
\end{definition}

This basis is related to the $\bK$-basis for $\Qsym ^n$ known as the basis of \emph{quasisymmetric Schur functions} $\{ \qs _\alpha \} _{\alpha \vDash n}$ \cite{BLvW, HLMvW-1, HLMvW-2} by the involutive automorphism of $\Qsym$ that maps $F_\alpha \mapsto F_{\reverse{\alpha}}$, which maps
$\qsy _\alpha \mapsto \qs _{\reverse{\alpha}}$. Thus, quasisymmetric Schur functions can be defined combinatorially as follows. 

\begin{definition}\label{def:qs} Let $\alpha \vDash n$. Then the \emph{quasisymmetric Schur function} $\qs _\alpha$ is given by $\qs _0 =1$ and otherwise
\begin{equation}\label{eq:qs}\qs _\alpha = \sum _\beta \hat{d}_{\reverse{\alpha} \reverse{\beta}} F_\beta\end{equation}where the sum is over all compositions $\beta \vDash  n$ and $\hat{d}_{\alpha\beta}=$ the number of $SYCTs$ $\tau$ of shape $\alpha$ such that $\des (\tau)=\set (\beta)$. 
\end{definition}

\begin{example}\label{ex:qsqsy} We have $\qsy _{(1,3)}= F_{(1,3)}$ from the SYCT $$\tableau{2&3&4\\1}$$while $\qs _{(1,3)}= F_{(1,3)}+F_{(2,2)}$ from the following SYCTs.
$$\tableau{4\\1&2&3}\qquad \tableau{3\\1&2&4}$$
\end{example}

 \begin{remark} These functions are named (Young) quasisymmetric Schur functions because they refine Schur functions in a natural way, that is,
 $$s_\lambda = \sum _{\partitionof{\alpha}=\lambda} \qsy _\alpha = \sum _{\partitionof{\alpha}=\lambda} \qs _\alpha$$and additionally refine many properties of Schur functions \cite{BLvW, HLMvW-1, HLMvW-2, LMvW-book}.
 \end{remark}

\section{Classical character theory}\label{sec:chars}

The basics of representation theory for finite groups can be found in texts such as \cite{curtis-reiner}, and for the symmetric groups in particular in texts such as \cite{sagan}.
For a finite group $G$, a finite dimensional representation of $G$ is a group homomorphism $\rho:G\to Gl(V)$ for some finite dimensional $\bK$-vector space $V$.
We can identify $\rho$ with $V$ viewed as a $G$-module under the  group action
\[ g\cdot v = \rho(g)(v) \qquad \forall \; g\in G,\, v\in V. \]
The \emph{character} of $\rho$ is the map $G\to\bK$ given by the trace function, $tr(\rho(g))$.
The set of irreducible characters of $G$ (that is, the characters of the irreducible modules) is denoted $\irr(G)$, and every character of $G$ is a positive $\bZ$-linear combination of irreducible characters.
Since non-isomorphic irreducible $G$-modules have distinct characters, $\irr(G)$ is in bijection with the set of isomorphism classes of irreducible $G$-modules.

The Grothendieck ring $\ralg(G)$ of finite dimensional representations of $G$ has as basis the set of isomorphism classes of irreducible $G$-modules, that is $\ralg(G)\cong\bK\,\irr(G)$ as a vector space,
with operations 
\[ [M]+[N]=[M\oplus N],  \qquad \text{and} \qquad  [M]*[N] = [M\otimes N]. \]
The algebra of class functions of $G$ is
\[ \calg(G) = \{ \phi:G\to\bK \suchthat \phi \text{ constant on conjugacy classes of } G \}, \]
under pointwise addition and multiplication.
The set $\irr(G)$ forms a $\bZ$-basis of $\calg(G)$.
In fact $\ralg(G)\cong\calg(G)$, the isomorphism being given by the trace function, $\tr:\ralg(G)\to\calg(G)$.

In the case of the symmetric groups, 
the conjugacy classes of $\sgrp_n$ are indexed by partitions $\lambda \vdash n$ corresponding to the cycle types of permutations.
Hence the irreducible  modules, called \emph{Specht} modules, $S^\lambda$ of $\sgrp_n$ are also indexed by partitions $\lambda \vdash n$.  We denote the irreducible characters by $\chi^\lambda = tr(S^\lambda)$.
There is a natural embedding of $\sgrp_m\times\sgrp_n$ into $\sgrp_{m+n}$, and so considering the direct sum of their representation rings,
\[ \ralg=\bigoplus_{n\ge 0} \ralg(\sgrp_n), \]
one can define a graded outer product on $\ralg$ using induction, that is,
 \[   [M]\cdot[N] = \left[\ind_{\sgrp_m\times\sgrp_n}^{\sgrp_{m+n}} M\otimes N \right] 
      \qquad \text{ for } [M] \in \ralg(\sgrp_m),\;[N] \in \ralg(\sgrp_n). \]
By identification via the trace function, this induces an isomorphic algebraic structure on the direct sum of class algebras, \[ \calg =\bigoplus_{n\ge 0} \calg(\sgrp_n). \]
It is well-known \cite[\S4.7]{sagan} that $\calg$ is isomorphic to $\Sym$, the algebra of symmetric functions, the isomorphism being given by the Frobenius characteristic map $\ch:\calg\to\Sym$ determined by \begin{equation}\label{eq:chmap}\ch(\chi^\lambda)=s_\lambda\end{equation}
that is, the irreducible character indexed by $\lambda$ is mapped to the Schur function indexed by $\lambda$.

The characteristic map can also be described in terms of \emph{Young characters} \cite[p. 6]{bless-schock}.
\label{sec:Walpha}
Let $\alpha =(\alpha_1,\ldots,\alpha_\ell)\vDash n$ and $\alpha=\comp(J)$, where $J\subseteq[n-1]$.
Letting $S=\{s_1,\ldots,s_{n-1}\}$ be the  set of standard generators for $W=\sgrp_n$ as a Coxeter group, we let $J^c$ be the complement of $J$ in $[n-1]$, 
and let $W_\alpha = W_{J^c}$ be the Young subgroup of $\sgrp_n$ generated by $\{s_j \suchthat j\in J^c\}$, which we identify with $\sgrp_{\alpha_1} \times \cdots\times \sgrp_{\alpha_\ell}$.
The Young character $\ych{\alpha}$ is defined to be the character of $\sgrp_n$ induced by the trivial character of $W_\alpha$, that is,
\[ \ych{\alpha}= \ind^{\sgrp_n}_{W_\alpha} 1 .\]
Note that $\ych{\alpha}=\ych{\beta}$ if $\partitionof{\alpha}=\partitionof{\beta}$, that is, if $\alpha$ and $\beta$ have the same underlying partition.
Furthermore,  $\ch(\ych{\alpha}) = h_\lambda$, where $\lambda=\partitionof{\alpha}$ and $h_\lambda$ is the complete homogeneous symmetric function indexed by $\lambda$ \cite[p. 114]{macdonald-1}.

\section{The permutation and descent algebras, and the Hopf algebra of noncommutative symmetric functions} \label{sec:palg}

For any Coxeter group $W$, Solomon \cite{solomon} defined a sublagebra  
of the group ring called the \emph{descent algebra of $W$}.
For the symmetric group $\sgrp_n$, we denote its group ring by $\palg_n= \bK\sgrp_n$, and its descent algebra $\dalg_n\subset \palg_n$.
To any $I\subseteq[n-1]$ we associate the element $\dah_I\in\palg_n$ (respectively, $\dar_I\in\palg_n$) that is the formal sum of all permutations in $\sgrp_n$ whose descent set is contained in (respectively, is equal to) $I$, that is,
\[  \dah_I = \sum_{\des(\sigma)\subseteq I} \sigma, \qquad \dar_I = \sum_{\des(\sigma)= I} \sigma. \]
Then $\{\dah_I\}_{I\subseteq[n-1]}$ and $\{\dar_I\}_{I\subseteq[n-1]}$ respectively are $\bZ$-bases of $\dalg_n$.  For ease of computation we write $\dah_\alpha = \dah_I$ and $\dar_\alpha = \dar_I$ when $n$ is understood from context and $\comp(I)=\alpha$.
If we let $W^\alpha$ denote the set of minimal length representatives of the left cosets of $W_\alpha$ in $\sgrp_n$, then 
note that $W^\alpha=\{\pi\in\sgrp_n \suchthat \des(\pi)\subseteq \set(\alpha) \}$. Thus $\dah_\alpha = \sum_{\sigma\in W^\alpha} \sigma$.

Using the convention that $\dalg_0=\palg_0=\bK$, we write
\[ \palg = \bigoplus_{n\geqslant 0} \palg_n, \qquad \text{and} \qquad  \dalg = \bigoplus_{n\geqslant 0} \dalg_n. \]

In relation to the representation theory of the $0$-Hecke algebra we   have the Hopf algebra of noncommutative symmetric functions, and a summary of connections can be found in \cite{thibon-lect}. We describe it here in order to both algebraically and combinatorially define two key objects of study, noncommutative Schur functions and Young noncommutative Schur functions. 

The Hopf algebra of \emph{noncommutative symmetric functions}\cite{GKLLRT} is a graded Hopf algebra
$$\Nsym = \bigoplus _{n\geqslant 0} \Nsym ^n \cong \bK\langle \nce_1, \nce_2, \ldots \rangle$$where we set $\nce _0=1$ and the $\nce _r$ for $r>0$ are noncommuting indeterminates of degree $r$. We call $\nce _r$ the \emph{$r$-th noncommutative elementary symmetric function}. Defining the \emph{$r$-th noncommutative complete homogeneous symmetric function} to be
$$\nch _r=\sum _{(\beta_1,\beta_2, \ldots , \beta_m) \vDash r} (-1)^{m-n} \nce _{\beta_1}\nce _{\beta_2} \cdots \nce _{\beta_m}$$Gelfand et al. showed in \cite{GKLLRT} that
$$\Nsym \cong \bK\langle \nch_1, \nch_2, \ldots \rangle$$and that for a composition $\alpha=(\alpha _1, \ldots , \alpha _\ell)$ if
$$\nce_\alpha = \nce_{\alpha _1}\cdots \nce_{\alpha _\ell}$$
$$\nch_\alpha = \nch_{\alpha _1}\cdots \nch_{\alpha _\ell}$$then $\Nsym ^n$ for $n>0$ has either $\{\nce_\alpha\} _{\alpha \vDash n}$ or $\{\nch_\alpha\} _{\alpha \vDash n}$ as a $\bK$-basis. Another basis for $\Nsym ^n$ is the basis of \emph{noncommutative ribbon Schur functions} $\{\ncr_\alpha\} _{\alpha \vDash n}$, defined for every composition $\alpha=(\alpha _1, \ldots , \alpha _\ell)$ to be
$$\ncr _\alpha = (-1) ^\ell \sum _{\alpha \refines (\beta _1, \ldots , \beta _k)}(-1)^k \nch _{\beta _1} \cdots \nch _{\beta _k}.$$It was shown in \cite[Section 6]{GKLLRT} that $\Qsym$ is the graded Hopf  dual to $\Nsym$ via the pairing of dual bases
\begin{equation}\label{eq:dual}
\langle M_\alpha, \nch _\beta \rangle = \langle F_\alpha, \ncr _\beta \rangle = \delta _{\alpha\beta}
\end{equation}where $\delta _{\alpha\beta} = 1$ if $\alpha = \beta$ and $0$ otherwise.
Meanwhile in \cite{malv-reut-1} it was shown that $\Qsym$ is the graded Hopf algebra dual to $\dalg$. Therefore $\Nsym$ and $\dalg$ are isomorphic. One such isomorphism 
$$\psi: \dalg \to \Nsym$$is
given by $\psi(\dah _\alpha)= \nch _\alpha$, or equivalently $\psi(\dar _\alpha) = \ncr _\alpha$. There is also a homomorphism of algebras
$$\Theta: \Nsym \to \Sym$$called the \emph{forgetful map}, which is given by $\Theta (\nch _r)=h_r$.

In \cite{BLvW} the authors introduced a new $\bZ$-basis for $\Nsym ^n$ for $n>0$, termed the basis of {noncommutative Schur functions}, which we denote by $\{\ncs_\alpha\} _{\alpha\vDash n}$ and can be defined algebraically as follows.

\begin{definition}\cite[Definition 5.5.1]{LMvW-book}\label{def:dualqs} Let $\alpha,\beta \vDash n$. Then the \emph{noncommutative Schur function} $\ncs _\beta$ is defined by
\begin{equation}\label{eq:dualqs}
\langle \qs _\alpha, \ncs _\beta \rangle = \delta _{\alpha\beta}.
\end{equation}
\end{definition}
There is an involutive anti-automorphism of $\Nsym$ which maps $\ncr_\alpha\mapsto\ncr_{\reverse{\alpha}}$.
We define the image of the basis $\{\ncs_\alpha\} _{\alpha\vDash n}$ under this involution by
 $\ncs_{\reverse{\alpha}} \mapsto \ncsy_\alpha$ and note by construction that this new basis $\{\ncsy_\alpha\} _{\alpha\vDash n}$, termed the basis of Young noncommutative Schur functions, can be defined algebraically as follows.
 
\begin{definition}\cite[Definition 5.6.1]{LMvW-book}
Let $\alpha,\beta \vDash n$. Then the \emph{Young noncommutative Schur function} $\ncsy _\beta$ is defined by
\begin{equation}\label{eq:dualqsy}
\langle \qsy _\alpha, \ncsy _\beta \rangle = \delta _{\alpha\beta}.
\end{equation}
\end{definition}

We can define both these types of functions combinatorially implicitly as every noncommutative ribbon Schur function  is a linear combination of (Young) noncommutative Schur functions with nonnegative integer coefficients. Specifically, if $\alpha \vDash n$, then by Equation~\eqref{eq:dual} and Equation~\eqref{eq:qsy}, respectively Equation~\eqref{eq:dual} and Equation~\eqref{eq:qs}, we have
\begin{equation}  \label{eq:Rtosr}
 \ncr_\alpha =  \sum_{\beta \vDash n}  \hat{d}_{\beta\alpha} \ncsy_\beta = \sum_{\beta \vDash n}  \hat{d}_{\reverse{\beta}\reverse{\alpha}} \ncs_\beta,
\end{equation}
where 
\begin{equation}
 \hat{d}_{\beta\alpha} = \# \{ \tau\in SYCTs \suchthat sh(\tau)= \beta,\, \comp(\tau)= \alpha \}.
\end{equation}

These functions are pre-images of  Schur functions under the forgetful map, that is by  \cite[Equation (5.8)]{LMvW-book} and \cite[Equation (2.12)]{BLvW},
\begin{equation} \label{eq:ThetaNcsr}
\Theta(\ncsy_\alpha)=\Theta(\ncs_\alpha)=s_{\partitionof{\alpha}}.
\end{equation}

%
%
\section{Noncommutative character theory}\label{sec:nc-char}

In his original paper \cite{solomon} where descent algebras were defined, Solomon showed that there is a ring homomorphism $\theta_W$ from the descent algebra of a finite Coxeter group $W$ into $\calg(W)$.
In the case of type A, $\theta_n:\dalg_n\to\calg(\sgrp_n)$ is surjective; specificallly, $\theta_n(\dah _\alpha)=\ych{\alpha}$ for $\alpha\vDash n$.

The maps $\theta_n$ collectively extend to a homomorphism of graded algebras  $\theta:\dalg\to\calg$ \cite[Thm. 1.2, Cor. 7.6]{bless-schock}.
Now for a composition $\alpha$
$$\ch (\theta ( \dah _\alpha )) = \ch (\ych{\alpha}) = h _{\partitionof{\alpha}} = \Theta (\nch _\alpha) = \Theta(\psi (\dah _\alpha ))$$that is, 
\begin{equation}\label{eq:mapcomm}\theta =  \ch^{-1}\circ \Theta\circ\psi .\end{equation}This can be described by the following commutative diagram.

\[ \begin{CD}
   \dalg @> \theta >>  \calg \\
   @V\psi  VV   @VV \ch V \\
   \Nsym @> \Theta>> \Sym 
   \end{CD} \]

\medskip
The map $\theta$ extends to a homomorphism of graded algebras $\theta:\palg\to\calg$.   
We refer the reader to \cite[Ch. 7]{bless-schock} for the details.
Recall that the irreducible characters $\{\chi^\lambda\}$ of the symmetric groups form a $\bZ$-basis of $\calg$,
and every character of a symmetric group is a positive $\bZ$-linear combination of irreducible characters. 
We now come to a key definition.

\begin{definition}\cite[p. 13]{bless-schock}\label{def:noncommchar}
Any inverse image under $\theta$ in $\palg$ of a character $\chi$ afforded by the $\sgrp_n$-module $M$ is called a \emph{noncommutative character} corresponding to $\chi$, or $M$.
\end{definition}

At this point we have  collected all our tools and made further useful deductions in the process. As a consequence, our main theorem is now straightforward to prove.

\begin{theorem}\label{the:noncommirrchars} The bases of $\Nsym ^n$ of Young noncommutative Schur functions $\{\ncsy _\alpha\} _{\alpha \vDash n}$ and of noncommutative Schur functions $\{\ncs _\alpha\} _{\alpha \vDash n}$ are each  isomorphic to a set of noncommutative irreducible characters of $\sgrp _n$. 

More precisely, every element of the set
$$\{ \psi ^{-1}(\ncsy _\alpha) \suchthat \alpha \vDash n \}$$is a noncommutative irreducible character of $\sgrp _n$. Similarly,  every element of the set
$$\{ \psi ^{-1}(\ncs _\alpha) \suchthat \alpha \vDash n \}$$is a  noncommutative irreducible character of $\sgrp _n$. 
\end{theorem}

\begin{proof} Let $\alpha \vDash n$, $\lambda = \partitionof{\alpha} \vdash n$ and $\chi ^\lambda$ be the corresponding irreducible character of $\sgrp _n$.  By Equation~\eqref{eq:chmap} 
$$\ch (\chi ^\lambda) = s_\lambda$$ and by Equation~\eqref{eq:ThetaNcsr}  
$$\Theta(\ncsy _\alpha)= s_\lambda.$$Therefore, since by Equation~\eqref{eq:mapcomm} $\theta =  \ch^{-1}\circ \Theta\circ\psi$, and so
$\theta \circ \psi ^{-1} = \ch ^{-1} \circ \Theta$ we have that
$$\theta(\psi^{-1}(\ncsy _\alpha))=\ch ^{-1} (\Theta (\ncsy _\alpha))= \ch ^{-1} (s _\lambda) = \chi ^\lambda.$$
Similarly by Equation~\eqref{eq:chmap}   and Equation~\eqref{eq:ThetaNcsr}   we have that $$\theta(\psi^{-1}(\ncs _\alpha)) = \chi ^\lambda$$ and therefore the theorem follows by Definition~\ref{def:noncommchar}.
\end{proof}

\begin{remark}\label{rem:minset} Both these bases for $\Nsym ^n$ satisfy the strong property that \emph{every} basis element is isomorphic to a noncommutative irreducible character of $\sgrp _n$. One might ask the question of which bases $\{\ncb _\alpha\} _{\alpha \vDash n}$ of $\Nsym ^n$ satisfy  the weaker condition that some subset of the basis is isomorphic to a set of noncommutative irreducible characters of $\sgrp _n$. As we can see from the above proof a necessary and sufficient condition for this is that there exists for every partition $\lambda \vdash n$ at least one $\alpha \vDash n$ such that
$$\Theta (\ncb_\alpha) = s _\lambda.$$To this end, the ``immaculate'' basis of  $\Nsym ^n$, $\{I _\alpha\} _{\alpha \vDash n}$,  satisfies $\Theta (I_\lambda)=s_\lambda$ \cite[Corollary 3.26]{BBSSZ-1}, and so is an  example of a basis that satisfies this weaker condition.
\end{remark}

%

\bibliography{NSrepV0}

\end{document}